\pdfoutput=1
\documentclass[preprint,12pt]{elsarticle}%
\usepackage{amsfonts}
\usepackage{amsmath}
\usepackage{amssymb}
\usepackage{graphicx}%
\providecommand{\U}[1]{\protect\rule{.1in}{.1in}}
\newtheorem{theorem}{Theorem}
\newtheorem{acknowledgement}[theorem]{Acknowledgement}
\newtheorem{algorithm}[theorem]{Algorithm}

\newtheorem{definition}[theorem]{Definition}

\newtheorem{lemma}[theorem]{Lemma}

\newtheorem{proposition}[theorem]{Proposition}

\newenvironment{proof}[1][Proof]{\noindent\textbf{#1.} }{\ \rule{0.5em}{0.5em}}
\begin{document}
%
\begin{frontmatter}%


%

\title
{Stability in 3d of a sparse grad-div approximation of the Navier-Stokes equations}%

%

\author{William Layton and Shuxian Xu}%
%

\address
{Department of Mathemaics, University of Pittsburgh, Pittsburgh PA 15260; emails: wjl@pitt.edu and shx34@pitt.edu}%
%

\begin{abstract}%

Inclusion of a term $-\gamma\nabla\nabla\cdot u$, forcing $\nabla\cdot u$ to
be pointwise small, is an effective tool for improving mass conservation in
discretizations of incompressible flows. However, the added grad-div term
couples all velocity components, decreases sparsity and increases the
condition number in the linear systems that must be solved every time step. To
address these three issues various sparse grad-div regularizations and a
modular grad-div method have been developed. We develop and analyze herein a
synthesis of a fully decoupled, parallel sparse grad-div method of Guermond
and Minev with the modular grad-div method. Let $G^{\ast}=-diag(\partial
_{x}^{2},\partial_{y}^{2},\partial_{z}^{2})$ denote the diagonal of
$G=-\nabla\nabla\cdot$, and $\alpha\geq0$\ an adjustable parameter. The 2-step
method considered is%
\begin{align}
1  & :\frac{\widetilde{u}^{n+1}-u^{n}}{k}+u^{n}\cdot\nabla\widetilde{u}%
^{n+1}+\nabla p^{n+1}-\nu\Delta\widetilde{u}^{n+1}=f\text{ \& }\nabla
\cdot\widetilde{u}^{n+1}=0,\nonumber\\
2  & :\left[  \frac{1}{k}I+(\gamma+\alpha)G^{\ast}\right]  u^{n+1}=\frac{1}%
{k}\widetilde{u}^{n+1}+\left[  (\gamma+\alpha)G^{\ast}-\gamma G\right]
u^{n}.\nonumber
\end{align}
We prove its unconditional, nonlinear, long time stability in $3d$ for
$\alpha\geq0.5\gamma$. The analysis also establishes that the method controls
the persistent size of $||\nabla\cdot u||$ in general and controls the
transients in $||\nabla\cdot u||$ when $u(x,0)=0$ and $f(x,t)\neq0$ provided
$\alpha>0.5\gamma$. Consistent numerical tests are presented.%

\end{abstract}%
%

%

\end{frontmatter}%



\section{Introduction}

We present and prove the long time, nonlinear stability of a fully uncoupled,
modular, sparse grad-div (SGD) finite element methods (FEM), approximating the
incompressible Navier-Stokes equations (NSE)%
\[
u_{t}+u\cdot\nabla u+\nabla p-\nu\Delta u=f(x,t)\text{ and }\nabla\cdot u=0.
\]
The stability analysis also delineates how the method controls $||\nabla\cdot
u||$. Sparse grad-div methods are one slice of research on improving mass
conservation in finite element methods. The complementary slice, currently
giving strong results, uses exactly divergence-free elements. These two,
others and their interconnections are surveyed in \cite{JLMNR16}. The first
sparse grad-div method considered herein is from Guermond and Minev
\cite{GM17}, for which we sharpen their stability result. The second is a new
but natural synthesis with the modular grad-div method of \cite{FLR18}. The
flow domain $\Omega$ is a bounded open set in ${{{\mathbb{R}}}}^{3}$ with no
slip boundary conditions $u=0$ on $\partial\Omega$ . Here $u\in{{{\mathbb{R}}%
}}^{3}$ is the velocity, $p\in{{{\mathbb{R}}}}$ is the pressure, $\nu$ is the
kinematic viscosity, and $f\in{{{\mathbb{R}}}}^{3}$ is the external force. Let
$\gamma$\ denote the (preset) grad-div parameter. Following Olshanskii
\cite{O02}, the standard grad-div approximation (with a simple time
discretization for concreteness) is the space discretization of%
\begin{gather}
\frac{u^{n+1}-u^{n}}{k}+u^{n}\cdot\nabla u^{n+1}+\nabla p^{n+1}-\nu\Delta
u^{n+1}-\gamma\nabla\nabla\cdot u^{n+1}=f(t^{n+1}),\nonumber\\
\text{and}{{\text{ \ \ \ }}}\nabla\cdot u^{n+1}=0.
\end{gather}
If (as here) neither the boundary conditions nor the viscosity depends on the
fluid stresses, the added grad-div term is the only term coupling all velocity
components. For $\gamma$\ large, the condition number of the linear system
increases \cite{LM16}. Even for moderate $\gamma$, penalizing pointwise
violation of incompressibility and asking $\nabla\cdot u$\ to be orthogonal to
the pressure space has been observed to cause solver issues \cite{FLR18}. To
eliminate this coupling, reduce memory requirements, speed parallel solution
and improve the robustness of iterative methods for the resulting linear
system, several sparse grad-div methods have been devised. To specify the
variant considered herein, let $G$ denote $-\nabla\nabla\cdot$ and $G^{\ast}%
$\ to be the diagonal of $G$%

\[
G:=-%
\begin{bmatrix}
\partial_{xx} & \partial_{xy} & \partial_{xz}\\
\partial_{yx} & \partial_{yy} & \partial_{yz}\\
\partial_{zx} & \partial_{zy} & \partial_{zz}%
\end{bmatrix}
\text{ \& }G^{\ast}:=-%
\begin{bmatrix}
\partial_{xx} & 0 & 0\\
0 & \partial_{yy} & 0\\
0 & 0 & \partial_{zz}%
\end{bmatrix}
.
\]
The synthesis of a sparse grad-div method of Guermond and Minev \cite{GM17}
with the modular grad-div method of \cite{FLR18} is as follows. Suppressing
the space discretization, given $u^{n}$ two approximations, $\widetilde
{u}^{n+1}$ and $u^{n+1}$, at the next time step are calculated by
\begin{align}
1  & :\frac{\widetilde{u}^{n+1}-u^{n}}{k}+u^{n}\cdot\nabla\widetilde{u}%
^{n+1}+\nabla p^{n+1}-\nu\Delta\widetilde{u}^{n+1}=f\text{ \& }\nabla
\cdot\widetilde{u}^{n+1}=0,\nonumber\\
& and\label{eq:ModularSGD1}\\
2  & :\left[  \frac{1}{k}I+(\gamma+\alpha)G^{\ast}\right]  u^{n+1}=\frac{1}%
{k}\widetilde{u}^{n+1}+\left[  (\gamma+\alpha)G^{\ast}-\gamma G\right]
u^{n}.\nonumber
\end{align}
The linear solve in Step 2 uncouples into 3 smaller and constant in time
systems (one for each velocity component). For example, the first sub-system,
for the $x$ component of velocity, is%
\[
\left[  I-k(\gamma+\alpha)\frac{\partial^{2}}{\partial x^{2}}\right]
u_{1}^{n+1}=RHS_{1}\text{.}%
\]
With simple discretizations, structured meshes, mass lumping and axi-parallel
domains the above $3$ sub-systems can even be written as one tridiagonal solve
for the unknowns on each mesh line. The precise presentation, including the
FEM discretization of their space derivatives, is given in Section 2. The
condition number of the coefficient matrix of Step 2 is proven in the appendix
(under typical assumptions for this estimation) to have a condition number
that does not blow up as $\gamma+\alpha\rightarrow\infty$, but changes from
parabolic conditioning to elliptic conditioning:
\[
cond_{2}(A)\leq C\frac{h^{2}+k(\gamma+\alpha)}{1+k(\gamma+\alpha)}h^{-2}.
\]
The usual $L^{2}(\Omega)$\ norm is denoted $||\cdot||$. The following
summarizes the essential result.

\begin{theorem}
Let $\gamma\geq0$. The method (\ref{eq:ModularSGD1}) is unconditionally,
nonlinearly, long time stable in $3d$ if $\alpha\geq0.5\gamma$ and in $2d$ if
$\alpha\geq0$.\linebreak\ Further, if $u^{0},\nabla u^{0}\in L^{2}(\Omega)$
and $||f(t^{n})||\leq C<\infty$ and $\alpha>0.5\gamma$\ we have $\nabla\cdot
u^{n}\rightarrow0$ as $\gamma\rightarrow\infty$ in time-average
\[
\lim\sup_{N\rightarrow\infty}\frac{1}{N}\sum_{n=1}^{N}||\nabla\cdot
u^{n}||^{2}\leq C\gamma^{-1}.
\]
If $u^{0}=0$, then for all $N$%
\[
\frac{1}{N}\sum_{n=0}^{N}||\nabla\cdot u^{n}||^{2}\leq C\gamma^{-1}.
\]

\end{theorem}

\subsection{Related work}

It has been recognized for a while now that the usual velocity-pressure FEM
can result in $\mathcal{O}(1)$ errors in mass conservation, $||\nabla\cdot
u||=\mathcal{O}(1)$, e.g., John, Linke, Merdon and Neilan \cite{JLMNR16} and
Belenli, Rebholz and Tone \cite{ART15}. This $||\nabla\cdot u||=\mathcal{O}%
(1)$ is clearly evident in the $\gamma=\alpha=0$\ tests in Section 5. The cure
for this is added grad-div stabilization, a simple idea with strong positive
consequences. Its origin seems to be in SUPG type local residual stabilization
methods, Brooks and Hughes \cite{BH82}, and the idea of adding an operator
positive definite on the constraint set in optimization. Detailed analysis of
the discretization, including an added grad-div term, can be found in papers
including Case, Ervin, Linke and Rebholz \cite{CE11}, Olshanskii and Reusken
\cite{OR04}, Olshanskii \cite{O02}, Jenkins, John, Linke and Rebholz
\cite{JJLR14}, Braack, B\"{u}rman, John and Lube \cite{BBJL07}, Layton,
Manica, Neda, Olshanskii and Rebholz \cite{LMNOR12}, Galvin, Linke, Rebholz
and Wilson \cite{GLRW12} and Connors, Jenkins and Rebholz \cite{CJR11}.
Preselection of the grad-div parameter $\gamma$\ is treated in many places
such as Heavner \cite{H17} and self-adaptive selection recently in Xie
\cite{X21}.

Linke and Rebholz \cite{LR13} developed the first sparse grad-div method.
Their method contributes no consistency error. It improves solver performance
\cite{BLR14}, \cite{LR13}, reducing coupling (in $3d$) from 3 components to 2
components followed sequentially by a 1 component solve. Since Linke and
Rebholz achieve this with a modified pressure, stability in $3d$ is automatic,
and higher-order time stepping is also available. Subsequent sparse grad-div
methods of Guermond and Minev \cite{GM17} achieved greater uncoupling at the
expense of increased consistency error and reduced options for time stepping.
Let $G$ denote $-\nabla\nabla\cdot$. Their first method selected $G^{\ast}%
$\ to be the upper triangular part of $G$ and lagged the remainder:
$\nabla\cdot u^{n+1}=0$\ and%
\begin{gather*}
\frac{u^{n+1}-u^{n}}{k}+u^{n}\cdot\nabla u^{n+1}+\nabla p^{n+1}-\nu\Delta
u^{n+1}+\\
\gamma G^{\ast}u^{n+1}-\gamma\lbrack G^{\ast}-G]u^{n}=f(t^{n+1}).
\end{gather*}
This method, sequentially uncoupling velocity components, was proved stable in
$2d$ and observed but not proven stable in $3d$. Their second sparse grad-div
method, equation (3.8) Section 3.3, uncoupled velocity components in parallel
as follows. For $\alpha$\ a free parameter select $G^{\ast}$\ to be the
diagonal of $G.$ The second method of Guermond and Minev \cite{GM17} is
$\nabla\cdot u^{n+1}=0$\ and%
\begin{gather}
\frac{u^{n+1}-u^{n}}{k}+u^{n}\cdot\nabla u^{n+1}+\nabla p^{n+1}-\nu\Delta
u^{n+1}+\nonumber\\
(\gamma+\alpha)G^{\ast}u^{n+1}-\left[  (\gamma+\alpha)G^{\ast}-\gamma
G\right]  u^{n}=f(t^{n+1}).\label{eq:SGD1}%
\end{gather}
In Theorem 3.3 they prove stability in $3d$ for $\alpha\geq2\gamma.$ The proof
given in Section 2 for the modular sparse grad-div method yields the following
sharpening of their stability result.

\begin{theorem}
Under the same conditions as Theorem 1, the conclusions of Theorem 1 for
method (\ref{eq:ModularSGD1}) hold as well for method (\ref{eq:SGD1}).
\end{theorem}

Modular (non-sparse) grad-div was introduced in \cite{FLR18}, where compared
to standard grad-div methods, dramatic reductions in run times and increases
in robustness were observed. Similar ideas for the grad-div operator were
developed in Minev and Vabishchevich \cite{MV}. Finding $\mathcal{O}(k^{2})$
extensions of (\textsc{SparseGD}) with the same unconditional stability is
nontrivial. The only step we are aware of (aside from defect/deferred
corrections wrapped around the first-order approximation used by Guermond and
Minev \cite{GM17}) is Trenchea\cite{T16}.

\section{Analysis of modular sparse grad-div}

This section makes the method and result precise and proves stability for
$\alpha\geq0.5\gamma$ and control of $\nabla\cdot u$\ for $\alpha>0.5\gamma$
for the modular sparse grad-div algorithm. This work builds on Guermond and
Minev \cite{GM17}, the work on modular grad-div in \cite{FLR18} and Rong and
Fiordilino \cite{RF20}, and the numerical tests of a related method in Demir
and Kaya \cite{DK19}. We suppress the traditional sub- or super-scripts "$h$"
in finite element formulations. Let $(X,Q)\subset\left(  \mathring{H}%
^{1}(\Omega)^{3},L_{0}^{2}(\Omega)\right) $ denote conforming, div-stable FEM
velocity-pressure spaces. To simplify the notation, define the following
bilinear forms and semi-norms.

\begin{definition}
In $3d$ (with the obvious modification for $2d$), define the symmetric
bilinear forms%
\[%
\begin{array}
[c]{ccc}%
\mathcal{A}(u,v) & := & (\gamma+\alpha)\left[  \left(  u_{1,x},v_{1,x}\right)
+\left(  u_{2,y},v_{2,y}\right)  +\left(  u_{3,z},v_{3,z}\right)  \right] ,\\
\mathcal{B}(u,v) & := & \mathcal{A}(u,v)-\gamma(\nabla\cdot u,\nabla\cdot
v),\\
\mathcal{B}^{\ast}(u,v) & := & \mathcal{B}(u,v)-\frac{1}{3}(\alpha
-2\gamma)(\nabla\cdot u,\nabla\cdot v).
\end{array}
\]
If $a(u,v)$ is a symmetric, positive semi-definite bilinear form on $X$ we
denote its induced semi-norm by $||v||_{a}^{2}=a(v,v).$
\end{definition}

The nonlinear term below has been explicitly skew-symmetrized and treated
linearly implicitly below. Other choices are possible within the analysis we
present, such as the EMAC formulation \cite{OR20}.

\begin{algorithm}
\lbrack Modular SGD]Given the initial velocity $u^{0}$ and grad-div parameter
$\gamma>0$, choose $\alpha\geq0$.

Step 1: Given $u^{n}\in X$, find $(\tilde{u}^{n+1},p^{n+1})\in(X,Q)$, for all
$(v,q)\in(X,Q)$ satisfying:%
\begin{gather*}
\left(  \frac{\tilde{u}^{n+1}-{\ u}^{n}}{k},v\right)  +\frac{1}{2}{\ }\left(
{u}^{n}\cdot\nabla\tilde{u}^{n+1},v\right)  -\frac{1}{2}{\ }\left(  {u}%
^{n}\cdot\nabla v,\tilde{u}^{n+1}\right) \\
+\nu\left(  \nabla\tilde{u}^{n+1},\nabla v\right)  -\left(  p^{n+1}%
,\nabla\cdot v\right)  =\left(  f^{n+1},v\right)  ,\\
\text{and }\left(  \nabla\cdot\tilde{u}^{n+1},q\right)  =0.
\end{gather*}

Step 2: Given $\tilde{u}^{n+1}\in X$, find $u^{n+1}\in X$, for all $v\in X$
satisfying:%
\[
(u^{n+1},v)+k\mathcal{A}(u^{n+1},v)=(\tilde{u}^{n+1},v_{h})+k\mathcal{B}%
(u^{n},v).
\]

\end{algorithm}

Step 1 uses the standard implicit method to calculate $\tilde{u}^{n+1}$ at
$t=t^{n+1}$. Step 2 adds the sparse grad-div stabilization term to $\tilde
{u}^{n+1}$ to get $u^{n+1}$. This separation of velocity approximations to one
where $\nabla\cdot u{\Large \bot}Q$ and one where $||\nabla\cdot u||$ is small
may be a reason for the increased robustness observed in \cite{FLR18}. For all
time steps, the uncoupled, same block diagonal matrix $I+k(\gamma
+\alpha)G^{\ast}$ arises in Step 2.

We begin with a lemma.

\begin{lemma}
Let $\gamma>0,\alpha\geq0$, then%
\[
\mathcal{B}(v,v)\geq\frac{\alpha-2\gamma}{3}||\nabla\cdot v||^{2}.
\]
Thus, if $\alpha\geq2\gamma>0$\ then $\mathcal{B}(v,v)\geq0$ for all $v\in
X$.\newline Otherwise, if $\alpha\geq0$ then, for all $v\in X,$%
\begin{equation}
\mathcal{B}^{\ast}(v,v)=\mathcal{B}(v,v)-\frac{1}{3}(\alpha-2\gamma
)||\nabla\cdot v||^{2}\geq0.\label{eq:BstarPositive}%
\end{equation}

\end{lemma}

\begin{proof}
The second and third claim follow from the first. For the first, since
$||\nabla\cdot v||^{2}=||v_{1,x}+v_{2,y}+v_{3,z}|||^{2}$ $\leq3||v_{1,x}%
||^{2}+3||v_{2,y}||^{2}+3||v_{3,z}||^{2}$, we have
\begin{align*}
\mathcal{B}(v,v)  & =(\gamma+\alpha)\left[  ||v_{1,x}||^{2}+||v_{2,y}%
||^{2}+||v_{3,z}||^{2}\right]  -\gamma||\nabla\cdot v||^{2}\\
& =(\gamma+\alpha)\left[  ||v_{1,x}||^{2}+||v_{2,y}||^{2}+||v_{3,z}%
||^{2}\right]  -\frac{\gamma+\alpha}{3}||\nabla\cdot v||^{2}\\
& -\left[  \gamma-\frac{\gamma+\alpha}{3}\right]  ||\nabla\cdot v||^{2}\\
& \geq(\gamma+\alpha)\left[  ||v_{1,x}||^{2}+||v_{2,y}||^{2}+||v_{3,z}%
||^{2}\right]  -\\
& \frac{\gamma+\alpha}{3}\left[  3||v_{1,x}||^{2}+3||v_{2,y}||^{2}%
+3||v_{3,z}||^{2}\right]  +\left[  \frac{\alpha-2\gamma}{3}\right]
||\nabla\cdot v||^{2}\\
& \geq\frac{\alpha-2\gamma}{3}||\nabla\cdot v||^{2}\geq0.
\end{align*}

\end{proof}

For all cases in the following theorem, the stability is proven via a formula
like
\[
E^{n+1}-E^{n}+2kD^{n+1}\leq2k(f,\widetilde{u}^{n+1}),
\]
which immediately implies stability (by summing over $n=1,N$) provided the
dissipation $D\geq0$ and the energy $E$\ is square of a norm of $u$. The $2d $
result and the one below for $\alpha\geq2\gamma$ in $3d$ are noted by Guermond
and Minev \cite{GM17} for method (\ref{eq:SGD1}).

\begin{proposition}
Consider the modular sparse grad-div method.

\textbf{2d case:} Assume $\gamma\geq0$. In $2d$ it is unconditional stable
when $\alpha\geq0$:
\begin{gather*}
\ \left[  ||u^{n+1}||^{2}+k\gamma(||u_{1,x}^{n+1}||^{2}+||u_{2,y}^{n+1}%
||^{2}\right] \\
-\left[  \Vert u^{n}\Vert^{2}+k\gamma(||u_{1,x}^{n}||^{2}+||u_{2,y}^{n}%
||^{2}\right]  +\\
\Vert\tilde{u}^{n+1}-u^{n}\Vert^{2}+||u^{n+1}-\tilde{u}^{n+1}||^{2}+2k\nu
\Vert\nabla\tilde{u}^{n+1}\Vert^{2}\leq2k(f^{n+1},\tilde{u}^{n+1}).
\end{gather*}

\textbf{3d case:} Suppose $2\gamma>\alpha\geq0.5\gamma$, then in $3d$ it is
unconditionally stable. It satisfies%
\[
E^{n+1}-E^{n}+2kD^{n+1}=2k(f,\widetilde{u}^{n+1}),
\]
where%
\begin{align*}
E^{n+1}  & =||u^{n+1}||^{2}+2k\left[  \frac{1}{2}||u^{n+1}||_{\mathcal{B}%
^{\ast}}^{2}+\frac{2\gamma-\alpha}{6}||\nabla\cdot u^{n+1}||^{2}\right] ,\\
D^{n+1}  & =\nu\Vert\nabla\tilde{u}^{n+1}\Vert^{2}+\frac{1}{2k}\left[
\Vert\tilde{u}^{n+1}-u^{n}\Vert^{2}+||u^{n+1}-\tilde{u}^{n+1}||^{2}\right] \\
& +\frac{1}{2}||u^{n+1}-u^{n}||_{\mathcal{B}^{\ast}}^{2}+\frac{2}{3}\left(
\alpha-0.5\gamma\right)  ||\nabla\cdot u^{n+1}||^{2}\\
& +\frac{2\gamma-\alpha}{6}||\nabla\cdot(u^{n+1}+u^{n})||^{2}.
\end{align*}

If $\alpha\geq2\gamma$, then in $3d$ it is unconditionally stable. It
satisfies%
\[
E^{n+1}-E^{n}+2kD^{n+1}=2k(f,\widetilde{u}^{n+1}),
\]
where%
\begin{align*}
E^{n+1}  & =||u^{n+1}||^{2}+k||u^{n+1}||_{\mathcal{B}}^{2},\\
D^{n+1}  & =\nu\Vert\nabla\tilde{u}^{n+1}\Vert^{2}+\frac{1}{2k}\left[
\Vert\tilde{u}^{n+1}-u^{n}\Vert^{2}+||u^{n+1}-\tilde{u}^{n+1}||^{2}\right] \\
& +\gamma||\nabla\cdot u^{n+1}||^{2}+\frac{1}{2}||u^{n+1}-u^{n}||_{\mathcal{B}%
}^{2}.
\end{align*}

\textbf{Control of }$\nabla\cdot u$ \textbf{in 3d}: Suppose $\alpha>0.5\gamma
$, $u^{0},\nabla u^{0}\in L^{2}(\Omega)$\ and $||f(t^{n})||_{-1}\leq F<\infty$
. Then if $u^{0}=0$\ we have for any $N$
\[
\frac{1}{N}\sum_{n=1}^{N}||\nabla\cdot u^{n}||^{2}\leq C\gamma^{-1}.
\]
For $u^{0}$ non-zero we have $\nabla\cdot u^{n}\rightarrow0$ as $\gamma
\rightarrow\infty$ in the discrete time averaged sense
\begin{equation}
\lim\sup_{N\rightarrow\infty}\frac{1}{N}\sum_{n=1}^{N}||\nabla\cdot
u^{n}||^{2}\leq C\gamma^{-1}.
\end{equation}
If $\alpha=0.5\gamma$ the above results hold with $||\nabla\cdot u^{n}||^{2}$
replaced by $||\nabla\cdot(u^{n+1}+u^{n})||^{2}$.
\end{proposition}

\begin{proof}
\textbf{The 2d case: }To shorten the $2d$ proof we set $\alpha=0$%
.\textbf{\ }The idea of the proof in $2d$ is simple. We perform a basic energy
estimate and subsume the inconvenient terms in ones that fit the desired
pattern. Set $v=\tilde{u}^{n+1}$, $q=p^{n+1}$ in Step 1. Use the polarization
identity and multiply by $2k$. We obtain
\[
\Vert\tilde{u}^{n+1}\Vert^{2}-\Vert u^{n}\Vert^{2}+\Vert\tilde{u}^{n+1}%
-u^{n}\Vert^{2}+2k\nu\Vert\nabla\tilde{u}^{n+1}\Vert^{2}=2k(f^{n+1},\tilde
{u}^{n+1}).
\]
Take $v=u^{n+1}$ in Step 2, use the polarization identity, multiply by $2$ and
rearrange. We obtain
\begin{gather*}
||u^{n+1}||^{2}-||\tilde{u}^{n+1}||^{2}+||u^{n+1}-\tilde{u}^{n+1}||^{2}+\\
2k\left\{
\begin{array}
[c]{c}%
\gamma(\nabla\cdot u^{n},\nabla\cdot u^{n+1})\\
+\gamma\left[  (u_{1,x}^{n+1}-u_{1,x}^{n},u_{1,x}^{n+1})+(u_{2,y}%
^{n+1}-u_{2,y}^{n},u_{2,y}^{n+1})\right]
\end{array}
\right\}  =0.
\end{gather*}
Add the last two equations. We obtain
\begin{gather}
||u^{n+1}||^{2}-\Vert u^{n}\Vert^{2}+\Vert\tilde{u}^{n+1}-u^{n}\Vert
^{2}+||u^{n+1}-\tilde{u}^{n+1}||^{2}+2k\nu\Vert\nabla\tilde{u}^{n+1}\Vert
^{2}\nonumber\\
2k\left\{  \gamma(\nabla\cdot u^{n},\nabla\cdot u^{n+1})+\gamma\left[
(u_{1,x}^{n+1}-u_{1,x}^{n},u_{1,x}^{n+1})+(u_{2,y}^{n+1}-u_{2,y}^{n}%
,u_{2,y}^{n+1}) \right]  \right\} \label{eq:2dcase}\\
=2k(f^{n+1},\tilde{u}^{n+1}).\nonumber
\end{gather}
Expanding the term inside braces ($\{\cdot\}$) algebraically gives%
\begin{align*}
\left\{  \cdot\right\}   &  =\gamma\left[  ||u_{1,x}^{n+1}||^{2}%
+||u_{2,y}^{n+1}||^{2}\right]  -\gamma\left[  (u_{1,x}^{n+1},u_{1,x}%
^{n})+(u_{2,y}^{n+1},u_{2,y}^{n})\right] \\
&  +\gamma(\nabla\cdot u^{n},\nabla\cdot u^{n+1})\\
&  =\gamma\left[  ||u_{1,x}^{n+1}||^{2}+||u_{2,y}^{n+1}||^{2}\right]
+\gamma\left[  (u_{2,y}^{n+1},u_{1,x}^{n})+(u_{1,x}^{n+1},u_{2,y}^{n})\right]
.
\end{align*}
Using the Cauchy-Schwarz-Young inequality in the last line of the above, then
yields
\[
\left\{  \cdot\right\}  \geq\frac{\gamma}{2}\left[  ||u_{1,x}^{n+1}%
||^{2}+||u_{2,y}^{n+1}||^{2}\right]  -\frac{\gamma}{2}\left[  ||u_{1,x}%
^{n}||^{2}+||u_{2,y}^{n}||^{2}\right]  .
\]
Inserting this for the term in braces in (\ref{eq:2dcase}) then implies%
\begin{gather*}
\left[  ||u^{n+1}||^{2}+k\gamma\left[  ||u_{1,x}^{n+1}||^{2}+||u_{2,y}%
^{n+1}||^{2}\right]  \right] \\
-\left[  \Vert u^{n}\Vert^{2}+k\gamma\left[  ||u_{1,x}^{n}||^{2}+||u_{2,y}%
^{n}||^{2}\right]  \right] \\
+\Vert\tilde{u}^{n+1}-u^{n}\Vert^{2}+||u_{h}^{n+1}-\tilde{u}^{n+1}||^{2}%
+2k\nu\Vert\nabla\tilde{u}^{n+1}\Vert^{2}\leq2k(f^{n+1},\tilde{u}^{n+1}).
\end{gather*}
Stability now follows by subsuming the $\tilde{u}^{n+1}$\ on the RHS into
$2k\nu\Vert\nabla\tilde{u}^{n+1}\Vert^{2}$\ on the LHS and summing over $n $.

\textbf{The 3d case: }In $3d$ there are too many inconvenient terms to simply
use\ the Cauchy-Schwarz-Young inequality as in $2d$ to establish the energy
estimate. Set $v=\tilde{u}^{n+1}$, $q=p^{n+1}$ in Step 1, multiply by $2k$ and
use the polarization identity to get
\begin{equation}
\left[  \Vert\tilde{u}^{n+1}\Vert^{2}-\Vert u^{n}\Vert^{2}\right]
+\Vert\tilde{u}^{n+1}-u^{n}\Vert^{2}+2k\nu\Vert\nabla\tilde{u}^{n+1}\Vert
^{2}=2k(f^{n+1},\tilde{u}^{n+1}).\label{eq:Step1}%
\end{equation}
We note that Step 2 can be rewritten as%
\[
(u^{n+1},v)+k\gamma(\nabla\cdot u^{n+1},\nabla\cdot v)=(\tilde{u}%
^{n+1},v)-k\mathcal{B}(u^{n+1}-u^{n},v).
\]
Take $v=u^{n+1}$ in this form of Step 2, use the polarization identity,
multiply by $2$ and rearrange. We obtain
\begin{gather}
||u^{n+1}||^{2}-||\tilde{u}^{n+1}||^{2}+||u^{n+1}-\tilde{u}^{n+1}%
||^{2}+2k\gamma||\nabla\cdot u^{n+1}\mathbf{||}^{2}+\nonumber\\
+2k\mathcal{B}(u^{n+1}-u^{n},u^{n+1})=0.\label{eq:Step3}%
\end{gather}
Add equation (\ref{eq:Step1}) and (\ref{eq:Step3}). We obtain
\begin{gather*}
||u^{n+1}||^{2}-\Vert u^{n}\Vert^{2}+\Vert\tilde{u}^{n+1}-u^{n}\Vert
^{2}+||u^{n+1}-\tilde{u}^{n+1}||^{2}+2k\nu\Vert\nabla\tilde{u}^{n+1}\Vert
^{2}\\
+2k\left\{  \gamma||\nabla\cdot u^{n+1}||^{2}+\mathcal{B}(u^{n+1}%
-u^{n},u^{n+1})\right\}  =2k(f^{n+1},\tilde{u}^{n+1}).
\end{gather*}

\textbf{3d case with} $\alpha\geq2\gamma$\ . This case implies $\mathcal{B}%
(v,v)\geq0$. Apply the polarization identity to the $\mathcal{B}-$semi-inner
product and collect terms. This gives the following%
\begin{gather*}
\left[  ||u^{n+1}||^{2}+k||u^{n+1}||_{\mathcal{B}}^{2}\right]  -\left[  \Vert
u^{n}\Vert^{2}+k||u^{n}||_{\mathcal{B}}^{2}\right] \\
+\Vert\tilde{u}^{n+1}-u^{n}\Vert^{2}+||u^{n+1}-\tilde{u}^{n+1}||^{2}%
+2k\nu\Vert\nabla\tilde{u}^{n+1}\Vert^{2}\\
+2k\left\{  \gamma||\nabla\cdot u^{n+1}||^{2}+0.5||u^{n+1}-u^{n}%
||_{\mathcal{B}}^{2}\right\}  =2k(f^{n+1},\tilde{u}^{n+1}).
\end{gather*}
Summing over $n=1,N$ yields stability when $\alpha\geq2\gamma$.

\textbf{3d case with} $2\gamma>\alpha\geq0.5\gamma$. We thus focus on the term
in braces in the last equation. First, recall (\ref{eq:BstarPositive}),
$\mathcal{B}^{\ast}(v,v)\geq0.$ Thus $\mathcal{B}^{\ast}$\ induces a semi-norm
$||\cdot||_{\mathcal{B}^{\ast}}$ to which a polarization identity can be
applied. Motivated by this observation, rewrite algebraically the term in
braces as%
\begin{gather*}
\left\{  \gamma||\nabla\cdot u^{n+1}||^{2}+\mathcal{B}(u^{n+1}-u^{n}%
,u^{n+1})\right\}  =\\
=\left[  \mathcal{B}(u^{n+1}-u^{n},u^{n+1})-\frac{\alpha-2\gamma}{3}%
(\nabla\cdot\left(  u^{n+1}-u^{n}\right)  ,\nabla\cdot u^{n+1})\right]  +\\
+\left(  \gamma||\nabla\cdot u^{n+1}||^{2}+\frac{\alpha-2\gamma}{3}%
(\nabla\cdot\left(  u^{n+1}-u^{n}\right)  ,\nabla\cdot u^{n+1})\right)
:=\left[  I\right]  +\left(  II\right)
\end{gather*}
We expand and apply the polarization identity to the term in brackets,
$\left[  I\right]  $, giving%
\begin{align*}
\left[  I\right]   & =\mathcal{B}^{\ast}(u^{n+1}-u^{n},u^{n+1})\\
& =\frac{1}{2}\left(  ||u^{n+1}||_{\mathcal{B}^{\ast}}^{2}-||u^{n}%
||_{\mathcal{B}^{\ast}}^{2}+||u^{n+1}-u^{n}||_{\mathcal{B}^{\ast}}^{2}\right)
.
\end{align*}
Recall that $2\gamma>\alpha\geq0.5\gamma$, so that the multipliers are
non-negative, $2\gamma-\alpha>0$ and $\alpha-0.5\gamma\geq0$. The term in
parentheses, $\left(  II\right)  $, is expanded as%
\[
\left(  II\right)  =\frac{\alpha+\gamma}{3}||\nabla\cdot u^{n+1}||^{2}%
+\frac{2\gamma-\alpha}{3}(\nabla\cdot u^{n},\nabla\cdot u^{n+1}).
\]
Applying the polarization identity in the form $x\cdot y=-0.5(|x|^{2}%
+|y|^{2}-|x+y|^{2})$\ to the term $(\nabla\cdot u^{n},\nabla\cdot u^{n+1}%
)$\ gives
\begin{align*}
\left(  II\right)   & =\frac{\alpha+\gamma}{3}||\nabla\cdot u^{n+1}%
||^{2}-\frac{2\gamma-\alpha}{6}\left\{  ||\nabla\cdot u^{n+1}||^{2}%
+||\nabla\cdot u^{n}||^{2}\right\} \\
& +\frac{2\gamma-\alpha}{6}||\nabla\cdot(u^{n+1}+u^{n})||^{2}.
\end{align*}
This is rearranged algebraically to read%
\begin{align*}
\left(  II\right)   & =\frac{2\gamma-\alpha}{6}\left[  ||\nabla\cdot
u^{n+1}||^{2}-||\nabla\cdot u^{n}||^{2}\right]  +\frac{2}{3}\left(
\alpha-0.5\gamma\right)  ||\nabla\cdot u^{n+1}||^{2}\\
& +\frac{2\gamma-\alpha}{6}||\nabla\cdot(u^{n+1}+u^{n})||^{2}.
\end{align*}
Putting all this together, we then have%
\[
E^{n+1}-E^{n}+2kD^{n+1}=2k(f,\widetilde{u}^{n+1}),
\]
where%
\begin{align*}
E^{n+1}  & =||u^{n+1}||^{2}+2k\left[  \frac{1}{2}||u^{n+1}||_{\mathcal{B}%
^{\ast}}^{2}+\frac{2\gamma-\alpha}{6}||\nabla\cdot u^{n+1}||^{2}\right] ,\\
D^{n+1}  & =\nu\Vert\nabla\tilde{u}^{n+1}\Vert^{2}+\frac{1}{2k}\left[
\Vert\tilde{u}^{n+1}-u^{n}\Vert^{2}+||u^{n+1}-\tilde{u}^{n+1}||^{2}\right] \\
& +\frac{1}{2}||u^{n+1}-u^{n}||_{\mathcal{B}^{\ast}}^{2}+\frac{2}{3}\left(
\alpha-0.5\gamma\right)  ||\nabla\cdot u^{n+1}||^{2}\\
& +\frac{2\gamma-\alpha}{6}||\nabla\cdot(u^{n+1}+u^{n})||^{2}.
\end{align*}
Since all terms are non-negative, stability follows by summing over $n$.

\textbf{Control of }$\nabla\cdot u$: The subtlety in concluding control of
$||\nabla\cdot u||$ from stability is that $E^{0}$ \& $D^{n}$ both depend on
the grad-div parameter $\gamma$. For this reason we obtain control in a time
averaged sense. Bound the RHS of the energy inequality by
\[
2k(f,\widetilde{u}^{n+1})\leq k\nu||\nabla\tilde{u}^{n+1}\Vert^{2}+k\nu
^{-1}F^{2}%
\]
and subsume the first term in $D$. This implies
\[
E^{n+1}-E^{n}+kD^{n+1}\leq k\nu^{-1}F^{2}.
\]
Summing this over $n=0,...,N-1$, dividing by $N$ and dropping the nonnegative
$E^{N}$\ term gives%
\[
k\frac{1}{N}\sum_{n=1}^{N}D^{n}\leq\frac{1}{N}E^{0}+k\nu^{-1}F^{2}.
\]
The RHS is bounded uniformly in $N$ so the limit superior as $N\rightarrow
\infty$ of the LHS exists. We thus have%
\begin{gather*}
\lim\sup_{N\rightarrow\infty}\frac{1}{N}\sum_{n=1}^{N}D^{n}\leq\nu^{-1}%
F^{2}\text{ }\\
\text{and}\\
\frac{1}{N}\sum_{n=1}^{N}D^{n}\leq\nu^{-1}F^{2}\text{ if }u^{0}=0.
\end{gather*}
The claimed result now follows since $D$\ contains (with a positive
multiplier) the term $\gamma||\nabla\cdot u^{n+1}||^{2}$ if $\alpha
>0.5\gamma,$ and if $\alpha=0.5\gamma$ the term $||\nabla\cdot(u^{n+1}%
+u^{n})||^{2}$.
\end{proof}


\section{Stability and control of $\nabla\cdot u$\ for flow between 3d offset
cylinders}

We consider a $3d$ rotational flow obstructed by an offset cylindrical
obstacle inside a cylinder. Let $r_{1}=1,r_{2}=0.1,(c_{1},c_{2})=(0.5,0)$ and
\[
\Omega_{1}=\{(x,y):x^{2}+y^{2}<r_{1}^{2}\;\text{and}\;(x-c_{1})^{2}%
+(y-c_{2})^{2}>r_{2}^{2}\}.
\]
The domain is $\Omega$ $=$ $\Omega_{1}\times(0,1)$, a cylinder of radius and
height one with a cylindrical obstacle removed, depicted with the mesh used in
Figure \ref{Fig1}.
\begin{figure}[htp]
\begin{center}
\includegraphics[
height=2.5381in,
width=3.7933in
]{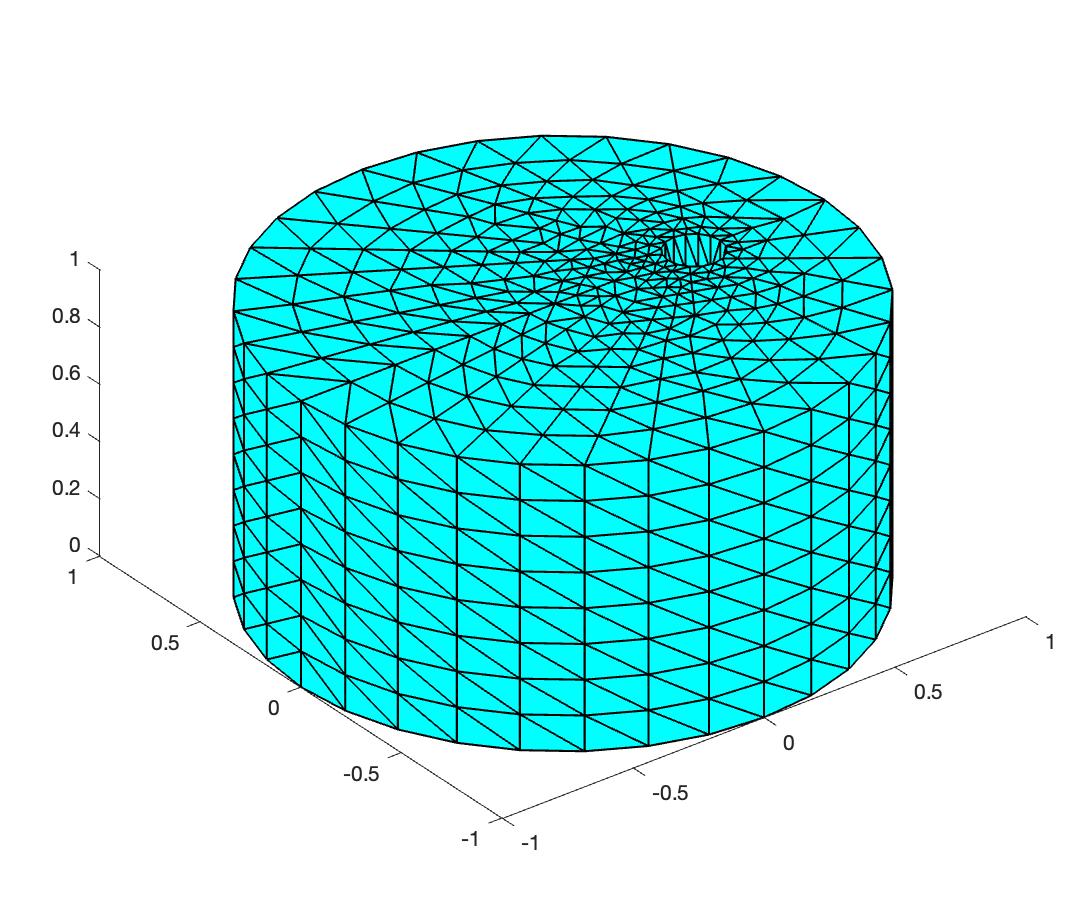}
\end{center}
\caption{Mesh used to test stability}%
\label{Fig1}%
\end{figure}

The flow is driven by a counter-clockwise rotational body force with $f=0$ on
the outer cylinder
\[
f(x,y,z,t)=\min\{t,1\}(-4y\ast(1-x^{2}-y^{2}),4x\ast(1-x^{2}-y^{2}%
),0)^{T},\;0\leq t\leq10.
\]
with no-slip boundary conditions, $u=0$, on boundaries. The space
discretization uses $P^{2}-P^{1}$ Taylor-Hood elements with $18972$ total
degrees of freedom in the velocity space and $2619$ total degrees of freedom
in the pressure space. This mesh in Figure \ref{Fig1} is insufficient to test
accuracy but suffices to test stability and control of $\nabla\cdot u$. The
flow rotates about the $z-$axis and interacts with the inner cylinder. We
start the test at rest, $u_{0}=(0,0,0)^{T}$, and choose the end time to be
$T=10$. The kinematic viscosity is $\nu=0.0001$ and the time step is $\Delta
t=0.05$.

We first tested if the extra $\alpha$\ term is necessary for stability. We
picked $\gamma=1.0,\alpha=0.5$ and $\alpha=0$ and $\gamma=\alpha=0$ (the
method with no grad-div term), solved and plotted the kinetic energy and
$||\nabla\cdot u||$ in Figure \ref{Fig2} below.
\begin{figure}[htp]
\begin{center}
\includegraphics[
height=2.5381in,
width=4.5933in
]{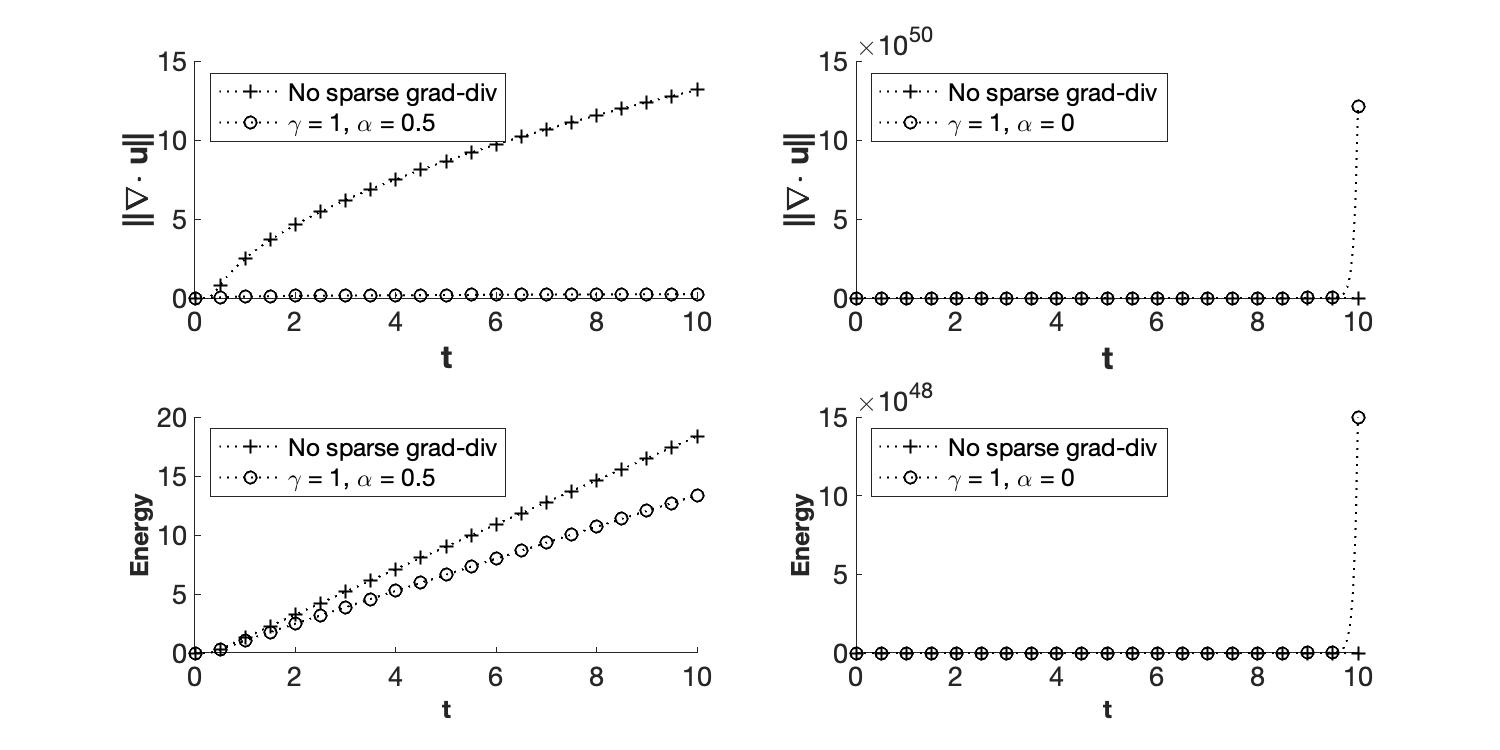}
\end{center}
\caption{Modular SGD. The left two plots are stable $\gamma$ and $\alpha$ pair
($\gamma=1, \alpha=0.5$) compared with no sparse grad-div term. The right two
are unstable $\gamma$ and $\alpha$ pair ($\gamma=1, \alpha=0$) compared with
no sparse grad-div term.}%
\label{Fig2}%
\end{figure}
The right hand side of the figure shows that the $\gamma=1.0,\alpha
=0.0$\ method is unstable while the $\gamma=1.0,\alpha=0.5$\ is stable. This
observed stability is consistent with the theoretical result.

The next question tested was whether $\alpha=0.5$ (for $\gamma=1$) is the
critical value for stability. To test this, we choose $\gamma=1$ and the range
of values $\alpha=0.3,0.4,0.48,0.49,0.5,0.6,0.7,1,2,3$, solved and plotted the
kinetic energy and $||\nabla\cdot u||$ vs time in Figure \ref{Fig3} and Figure
\ref{Fig4}.

\begin{figure}[htp]
\begin{center}
\includegraphics[
height=2.5381in,
width=4.5933in
]{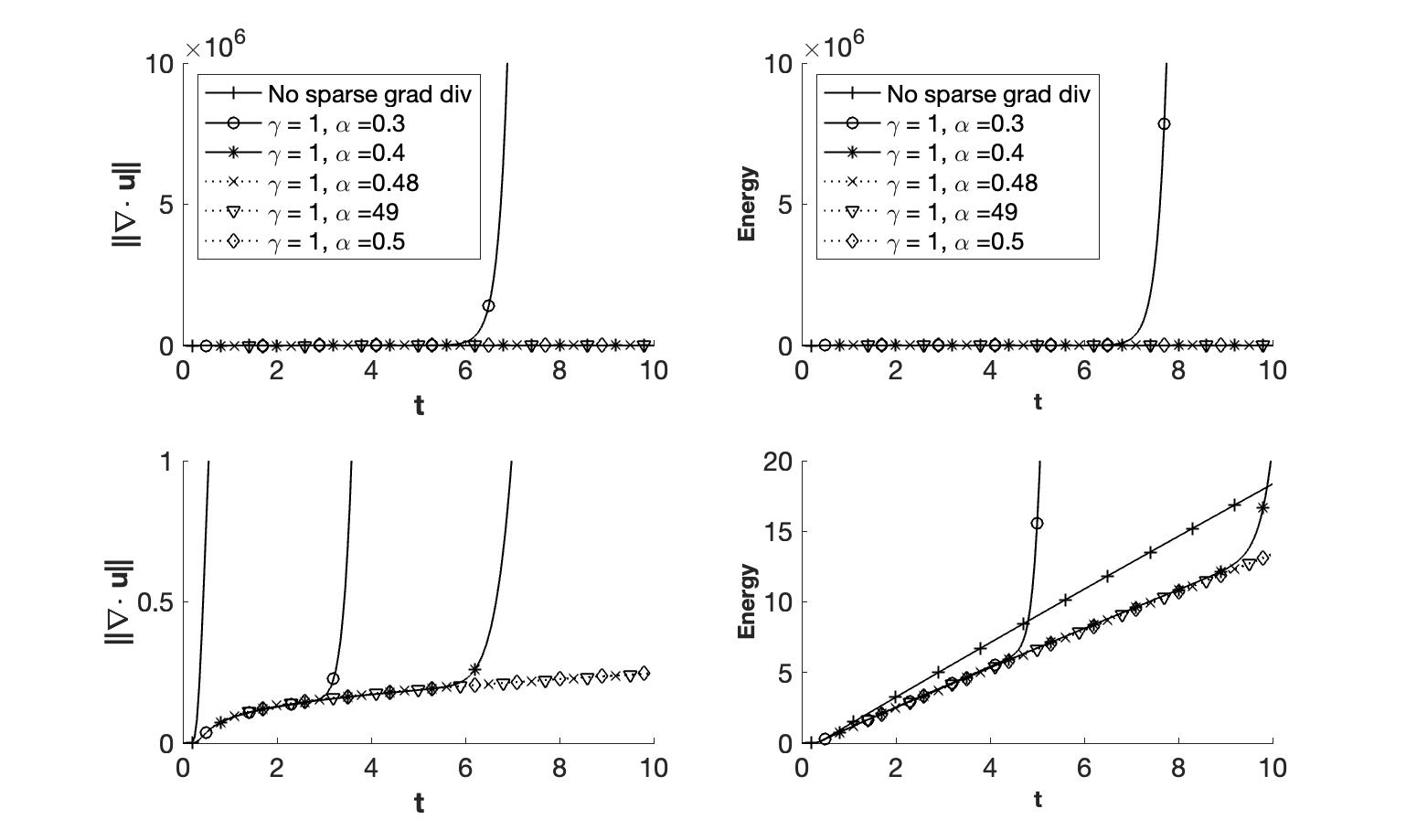}
\end{center}
\caption{Testing the $\alpha\leq0.5\gamma$, lower bound of $\alpha$ in
(\ref{eq:ModularSGD1}). The left two plots are $\|\nabla\cdot u\|$ vs time.
The right two plots are kinetic energy vs time. When $\alpha= 0.3, 0.4$,
results show instability.}%
\label{Fig3}%
\end{figure}
In Figure \ref{Fig4}, method (\ref{eq:ModularSGD1}) is stable for $\alpha
\geq0.5$, and in Figure \ref{Fig3}, for $\alpha<0.5$, the closer $\alpha$\ is
to the critical $0.5$ value, the longer time needed to see instability. No
instability over $0<t<10$ was observed for the nearly critical values
$\alpha=0.48$ \& $0.49.$ This could be because the time interval $0<t<10$ was
too short, because the derived value $\alpha=0.5$ is uniform in the viscosity
$\nu$, so actual stability is slightly better than proven or because some
sharpness was lost in the various inequalities. In further tests, we also
observe $\alpha=0.45$ instability starts near $t=21.5$. Similar behavior was
seen in the plots of $||\nabla\cdot u||$ in terms of control or loss of
control of $\nabla\cdot u$. The only evidence in the plots of $||\nabla\cdot
u||$ of non-sharpness of the analysis observed is that for $\alpha=0.5\gamma$,
control of $||\nabla\cdot u||$\ was observed. In contrast, the theorem
predicted control of averages over $2$ time levels for $\alpha=0.5\gamma$.
Please note that different scales were needed on the vertical axis.
\begin{figure}[htp]
\begin{center}
\includegraphics[
height=2.5381in,
width=4.5933in
]{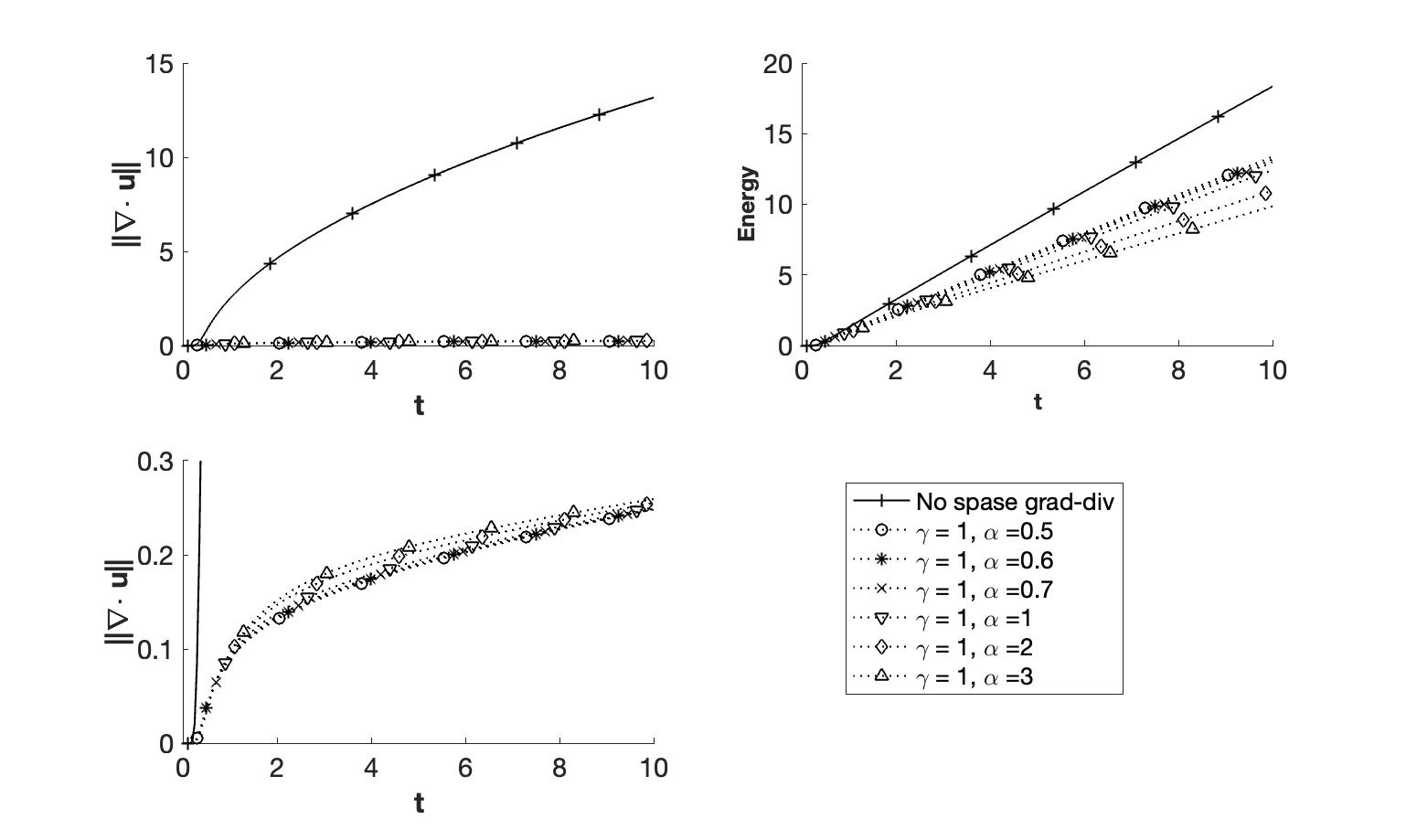}
\end{center}
\caption{Testing the $\alpha\geq0.5\gamma$, lower bound of $\alpha$ in
(\ref{eq:ModularSGD1}). The left two plots are $\|\nabla\cdot u\|$ vs time.
The right plot is kinetic energy vs time.}%
\label{Fig4}%
\end{figure}

Next, we compare the effect of $\gamma$ in (\ref{eq:ModularSGD1}) on
$||\nabla\cdot u||$. We choose $\gamma=0.1, 1,10,20,50,100$ and $\alpha
=0.5\ast\gamma$. For these values we solved and plotted the evolution of
$||\nabla\cdot u||$ and kinetic energy in Figure \ref{Fig5}.

The results in Figure \ref{Fig5} are consistent with $||\nabla\cdot u||$
decreasing as $\gamma$ increases. We also note that moderate values of
$\gamma$, e.g. $\gamma=0.1$ and $10$, in this test seem to be effective. We
conjecture that this is because $\nabla\cdot u=0$ is also required to be
orthogonal to the pressure space. We also present the time-average
$\|\nabla\cdot u\|^{2}$ and $\|\nabla\cdot u\|$ at end time $T=10$ for
different $\gamma$ in Table.\ref{tab:table1}. The convergence rate of average
$\|\nabla\cdot u\|^{2}$ about $-1$ is consistent with our analysis in the
control of $\nabla\cdot u$ in 3d.

We have also performed the above tests of (\ref{eq:SGD1}). The results were
similar so not detailed herein.
\begin{figure}[htp]
\begin{center}
\includegraphics[
height=2.5381in,
width=4.5933in
]{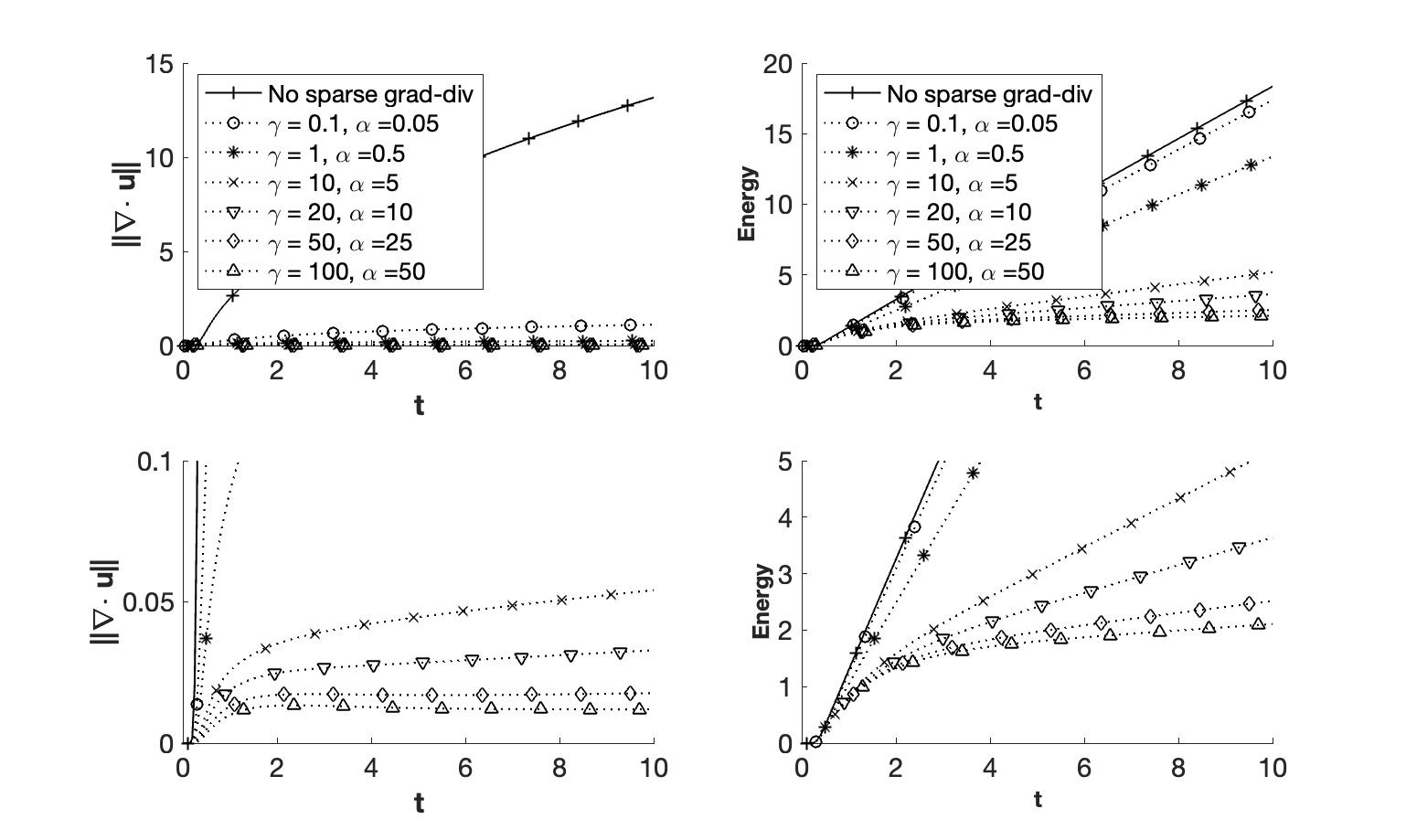}
\end{center}
\caption{Effect of $\gamma$ in (\ref{eq:ModularSGD1}) on velocity and
$\|\nabla\cdot u\|$. The left two plots are $\|\nabla\cdot u\|$ vs time. The
right two plots are energy vs time.}%
\label{Fig5}%
\end{figure}\begin{table}[pth]
\centering
\begin{tabular}
[c]{|c|c|c|c|c|}\hline
$\gamma$ & Avg($\|\nabla\cdot u\|^{2}$) & rate & $\|\nabla\cdot u(T)\|$ &
rate\\\hline
0.1 & 0.64305 & - & 1.1033 & -\\
1 & 0.033985 & -1.28 & 0.24826 & -0.65\\
10 & 0.0018455 & -1.27 & 0.054152 & -0.66\\
20 & 0.00074997 & -1.30 & 0.032871 & -0.72\\
50 & 0.00026663 & -1.13 & 0.017703 & -0.68\\
100 & 0.0001403 & -0.93 & 0.01195 & -0.57\\\hline
\end{tabular}
\caption{Time-average $\|\nabla\cdot u\|^{2}$ and $\|\nabla\cdot u\|$ at end
time $T$ for different $\gamma$ value when $\alpha= 0.5\gamma$.}%
\label{tab:table1}%
\end{table}


\section{Conclusions}

With $\alpha\geq0.5\gamma$ the algorithm presented is long time, nonlinearly
stable in $3d$ and fully uncouples all velocity components in the associated
linear system. For $\alpha<0.5\gamma$ the tests observed either instability or
loss of control of $\nabla\cdot u$\ (or both). The lower bound $0.5\gamma
$\ thus seems close enough to be sharp in the experiments to be useful. Open
problems include providing an analysis of stability in $3d$ for the sparse
grad-div method with $\alpha=0$ and $G^{\ast}$, the upper triangular part of
$G$, and for higher-order time discretizations.

\begin{acknowledgement}
The research presented herein was supported by NSF\ grant DMS 2110379.
\end{acknowledgement}

\section{Appendix: condition number estimation}

We give a brief analysis of the condition number of the coefficient matrix
occurring in Step 2: Given $\tilde{u}^{n+1}\in X$, find $u^{n+1}\in X$, for
all $v\in X$ satisfying:%
\[
(u^{n+1},v)+k\mathcal{A}(u^{n+1},v)=(\tilde{u}^{n+1},v_{h})+k\mathcal{B}%
(u^{n},v).
\]
As noted in the introduction, the coefficient matrix is block diagonal with
one block for each velocity \ component. Since all blocks have similar
structure and condition numbers, we estimate the condition number of the 1-1
block matrix. Let $\{\phi_{1},\cdot\cdot\cdot,\phi_{N}\}$\ denote a standard
finite element, nodal basis for the first component of the finite element
space, denoted $X_{1}$. Then the 1-1 block matrix we consider is%
\[
A_{ij}=(\phi_{i},\phi_{j})+k(\gamma+\alpha)(\frac{\partial}{\partial x}%
\phi_{i},\frac{\partial}{\partial x}\phi_{j}),i,j=1,\cdot\cdot\cdot,N.
\]
We assume the Poincar\'{e}-Friedrichs inequality holds in the x-direction,
$A1$ (excluding x-periodic boundary conditions) and make the following 2
standard assumptions, $A2, A3$, on $X_{1}$. These have been proven for many
spaces on quasi-uniform meshes.

A1 [Poincar\'{e}-Friedrichs]: For all $v\in X_{1},||v||\leq C||\frac{\partial
v}{\partial x}||.$

A2 [Inverse estimates]: For all $v\in X_{1},||\nabla v||\leq Ch^{-1}||v||.$

A3 [Norm equivalence]: We have $N=Ch^{-d},d=\dim(\Omega)=2$ or $3.$ For all
$v\in X_{1},v=\sum_{i=1}^{N}c_{i}\phi_{i}(x),||v||$ and $\sqrt{N^{-1}%
\sum_{i=1}^{N}c_{i}^{2}}$ are uniformly in $h$ equivalent norms.

For $|\cdot|$ the euclidean norm, we estimate $|A|$ and $|A^{-1}|$ below.
These two estimates show that%
\[
cond_{2}(A)\leq C\frac{1+k(\gamma+\alpha)Ch^{-2}}{1+k(\gamma+\alpha)}.
\]

For $|A^{-1}|$, let $Ac=b$\ then $|A^{-1}|^{2}=\max_{b}|A^{-1}b|^{2}/|b|^{2}$.
Let $M$ denote the finite element mass matrix $M_{ij}=(\phi_{i},\phi_{i})$.
Solve $Ma=b$. Define
\[
w=\sum_{i=1}^{N}c_{i}\phi_{i}(x),g=\sum_{i=1}^{N}a_{i}\phi_{i}(x).
\]
Then $Ac=b$ implies $w,g$ satisfy%
\[
(w,v)+k(\gamma+\alpha)(\frac{\partial}{\partial x}w,\frac{\partial}{\partial
x}v)=(g,v),\text{ for all }v\in X_{1}\text{.}%
\]
Setting $v=w$ and using A1 gives $\left(  1+k(\gamma+\alpha)C^{2}\right)
||w||\leq||g||$. Norm equivalence implies $||w||$ and $\sqrt{N^{-1}\sum
_{i=1}^{N}c_{i}^{2}}$ are uniformly in $h$ equivalent norms. Norm equivalence
applied twice implies $||w||$ and $\sqrt{N^{-1}\sum_{i=1}^{N}b_{i}^{2}}$ are
uniformly in $h$ equivalent norms. Thus
\[
|A^{-1}|\leq C\left(  1+k(\gamma+\alpha)\right)  ^{-1}.
\]
To estimate $|A|=\max_{c}|Ac|/|c|$, norm equivalence, A3, implies this is
equivalent to estimating above $||g||/||w||.$ We have $||g||=\max
_{v}(g,v)/||v||$. Then%
\begin{align*}
\frac{(g,v)}{||v||}  & =\frac{(w,v)+k(\gamma+\alpha)(\frac{\partial
w}{\partial x},\frac{\partial v}{\partial x})}{||v||}\\
& \leq||w||+k(\gamma+\alpha)\frac{||\frac{\partial w}{\partial x}%
||||\frac{\partial v}{\partial x}||}{||v||}\text{ and by A2}\\
& \leq||w||+k(\gamma+\alpha)Ch^{-2}||w||.
\end{align*}
Thus, $||g||/||w||\leq1+k(\gamma+\alpha)Ch^{-2}.$

\end{document}